\documentclass[a4paper, 10pt]{article}
\usepackage{amssymb}

\usepackage{amsmath,amssymb,amsthm}
\usepackage[latin1]{inputenc}
\usepackage{version,tabularx,multicol}
\usepackage{graphicx,float,psfrag}
\usepackage{stmaryrd}
 \usepackage{color}
\usepackage[pdftex]{hyperref}
\usepackage{amsfonts}
\usepackage{mathrsfs}
\usepackage{geometry}
\usepackage[numbers,sort&compress]{natbib}

\theoremstyle{plain}
\newtheorem{theorem}{Theorem}[section]

 \newtheorem{lemma}[theorem]{Lemma}

\newtheorem{remark}{Remark}[section]

\newtheorem{assumption}{Assumption}[section]

\newcommand{\ind}{{\bf 1}}

  \makeatletter
  \@addtoreset{equation}{section}
  \makeatother

 \def\beqlb{\begin{eqnarray}}\def\eeqlb{\end{eqnarray}}
 \def\beqnn{\begin{eqnarray*}}\def\eeqnn{\end{eqnarray*}}

\newcommand{\bcen}{\begin{center}}
\newcommand{\ecen}{\end{center}}
\newcommand{\bgeqn}{\begin{equation}}
\newcommand{\edeqn}{\end{equation}}

\begin{document}
\title{Lower deviation probabilities for level sets of the branching random walk}
\author{Shuxiong Zhang}
\maketitle

\renewcommand{\thefootnote}{\fnsymbol{footnote}}

\noindent\textit{Abstract:}
Given a branching random walk$\{Z_n\}_{n\geq0}$ on $\mathbb{R}$, let $Z_n([y,\infty))$ be the number of particles located  in $[y,\infty)$ at generation $n$. It is known from \cite{Biggins1977} that under some mild conditions, $n^{-1}\log Z_n([\theta x^* n,\infty))$ converges a.s. to $\log m-I(\theta x^*)$, where $\log m-I(\theta x^*)$ is a positive constant. In this work, we investigate its lower deviation, in other words, the convergence rates of
$$\mathbb{P}\left(Z_n([\theta x^* n,\infty))<e^{an}\right),$$
where $a\in[0,\log m-I(\theta x^*))$. Our results complete those in \cite{Mehmet}, \cite{Helower} and \cite{GWlower}.

\bigskip
\noindent Mathematics Subject Classifications (2020): 60F10, 60J80, 60G50.

\bigskip
\noindent\textit{Key words and phrases}: Branching random walk; Level sets; Lower deviation.

\section{Introduction and Main results}
\subsection{Background}
In this article, we study the lower deviation probabilities for local mass of the branching random walk (BRW). Given a probability distribution $\{p_k\}_{k\geq 0}$ on $\mathbb{N}$ and a real-valued random variable $X$, a branching random walk $\{Z_n\}_{n\geq 0}$ with offspring law $\{p_k\}_{k\geq 0}$ and step size $X$ is defined as follows. At time $0$, there is one particle located at the origin. The particle dies and produces  offsprings according to the offspring distribution $\{p_k\}_{k\geq 0}$. Afterwards, the offspring particles move independently according to the law of step size $X$, which forms a process $Z_1$. For any point process $Z_n$, $n\geq 2$, we define it by the following iteration
$$Z_n=\sum_{x\in Z_{n-1}}\tilde{Z}_1^{x},$$
where $\tilde{Z}_1^{x}$ has the same distribution as $Z_1(\cdot-S_x)$ and $\{\tilde{Z}_1^{x}: x\in Z_{n-1}\}$  (conditioned on $Z_{n-1}$)  are independent. Here and later, for a point process $\xi$, $x\in\xi$ means $x$ is an atom of $\xi$, and $S_x$ is the position of $x$ (i.e., $\xi=\sum_{x\in\xi}\delta_{S_x}$).\par
For $A\subset\mathbb{R}$, let
$$Z_n(A)=\#\left\{u\in Z_n: S_u\in A\right\},$$
i.e. the number of particles located in the set $A$.\par
Let $m:=Z_1(\mathbb{R})$. According to Biggins \cite[Theorem 2]{Biggins1977}: if $1<m<\infty$ and $\mathbb{E}[e^{\kappa X}]<\infty$ for some $\kappa>0$, then
for $\theta\in[0,1)$,
$$
\lim_{n\rightarrow\infty}\frac{1}{n}\log Z_n([\theta x^*n,\infty))=\log m-I(\theta x^*),~\mathbb{P}-a.s.~\text{on-non extinction},
$$
where $I(x):=\sup\limits_{t\in\mathbb{R}}\{tx-\log \mathbb{E}[e^{tX}]\}$, $x^*=\sup\{x\geq0:I(x)\leq\log m\}$. So, it's natural to study the decay rate of
$$\mathbb{P}(Z_n([\theta x^*n,\infty))< e^{an}),$$
where $a\in[0,\log m-I(\theta x^*))$.\par
 In fact, there are some related results for the branching Brownian motion $\{Z_t\}_{t\geq 0}$: \"{O}z \cite{Mehmet} considered the lower deviation $\mathbb{P}(Z_t(\theta x^*t+B)< e^{at})$, where $B$ is a fixed ball; A\"idekon, Hu and Shi \cite{levelsetzhan} considered the upper deviation $\mathbb{P}(Z_t([\theta x^*t,\infty))>e^{at})$. We will see that the for the branching random walk, the strategy to study this problem and the answers will be very different from theirs.\par
 We also  mention here that, since the last few decades, the model BRW has been extensively studied due to its connection to many fields, such as Gaussian multiplicative chaos, random walk in random environment, random polymer, random algorithms and discrete Gaussian free field etc; see \cite{hushi}, \cite{liu98}, \cite{liu06}, \cite{bramsonding15} and \cite{levelsetzhan} references therein. One can refer to Shi \cite{zhan} for a more detailed overview. The large deviation probabilities (LDP) for BRW and branching Brownian motion (BBM) on real line have attracted many researcher's attention. For example: Hu \cite{yhu}, Gantert and H{\"o}felsauer \cite{GWlower} and Chen and He \cite{Helower} considered the LDP and the moderate deviation of BRW's maximum (for BBM's maximum, see Chauvin and Rouault \cite{chauvin} and Derrida and Shi \cite{derrida16,derrida17,derrida172}); Chen, He \cite{ChenHe} and Louidor, Perkins \cite{Louidor} studied the large deviation of empirical distribution. Some other related works include Rouault \cite{rouault}, Buraczewski and Ma\'slanka  \cite{Burma} and Bhattacharya \cite{bhattacharya}.
\subsection{Main Results}
Before giving our results, we need first introduce some notations. Let $|Z_n|:=Z_n(\mathbb{R})$, $m:=\mathbb{E}[|Z_1|]$ and $b:=\min\{k\geq 0: p_k>0\}$. Recall that $\{p_k\}_{k\geq 0}$ is the offspring law, and $X$ is the step size. In the sequel of this work, we always need the following assumptions.
\begin{assumption}\label{assume}\leavevmode \\
(i) $\mathbb{E}[X]=0$, $\mathbb{P}(X=0)<1$ \text{and} $\mathbb{E}[e^{\kappa X}]<\infty$ for some $\kappa>0;$\\
(ii) $p_0=0$, $p_1<1$ and $\mathbb{E}[|Z_1|]<\infty;$\\
(iii) $a\in\left[0,\log m-I(\theta x^*)\right)$, $\theta\in[0,1).$
\end{assumption}
The first two theorems consider the Schr\"{o}der case (i.e., $p_1>0$).
\begin{theorem}\label{main1}(Schr\"{o}der case, light tail) Assume $p_1>0$, $\mathbb{E}[e^{\kappa X}]<\infty$ for some $\kappa<0$. Then
$$\lim_{n\rightarrow\infty}\frac{1}{n}\log \mathbb{P}\left(Z_n([\theta x^*n,+\infty))<e^{an}\right)
=-\inf\limits_{\rho\in(0,\bar{\rho}]}\left\{\rho\log\frac{1}{p_1}+\rho I\left(-\frac{d}{\rho}\right)\right\},$$
where $\bar{\rho}\in(0,1)$, $d\in[0,+\infty)$ are defined as follows:
\begin{align}
\bar{\rho}&:=\sup\left\{\rho\in(0,1):\log m-I\left(\frac{\theta x^*}{1-\rho}\right)-\frac{a}{1-\rho}\geq0\right\},\cr
d&:=d(\rho)=\sup\left\{h\in[0,+\infty):\log m-I\left(\frac{h+\theta x^*}{1-\rho}\right)-\frac{a}{1-\rho}\geq0\right\}.\nonumber
\end{align}
\end{theorem}
When the condition $\mathbb{E}[e^{\kappa X}]<\infty$ for some $\kappa<0$ fails, we will see, in the following theorem, the decay scale may depend on the tail of step size.
\begin{theorem}\label{main4}(Schr\"{o}der case, heavy tail) Assume $p_1>0$.\\
If $\mathbb{P}(X<-x)=\Theta(1)x^{-\alpha}$ as $x\rightarrow+\infty$ for some $\alpha>0$, then
$$\lim\limits_{n\rightarrow\infty}\frac{1}{\log n}\log\mathbb{P}\left(Z_n([\theta x^*n,\infty))<e^{an}\right)=-\alpha.$$
If $\mathbb{P}(X\leq-x)=\Theta(1)e^{-\lambda x^{\alpha}}$ as $x\rightarrow+\infty$ for some $\alpha\in(0,1)$ and $\lambda>0$, then
$$\lim\limits_{n\rightarrow\infty}\frac{1}{n^{\alpha}}\log\mathbb{P}\left(Z_n([\theta x^*n,\infty))<e^{an}\right)=
\begin{cases}
 -\lambda(1-\theta)x^*, &a\in[0,\log m-I(x^*)];\\
-\lambda\hat{c}, &a\in(\log m-I(x^*),\log m-I(\theta x^*)),
\end{cases}
$$
where $\hat{c}\in(0,x^{*}-\theta x^*)$ is the unique solution of
$$\log m-I(\theta x^*+\hat{c})=a.$$
\end{theorem}
\begin{remark} In fact, in the case $\mathbb{P}(X<-x)=\Theta(1)x^{-\alpha}$, the condition $\mathbb{E}[X]=0$ can be removed in Assumption \ref{assume}.
\end{remark}
The next two theorems consider the B\"{o}ttcher case (i.e. $p_1=0$). As we can see in the
following, the decay scale depends on the tail of step size $X$.
\begin{theorem}\label{main2}(B\"{o}ttcher case, bounded tail) Assume $p_1=0$, $ess\inf X=-L\in(-\infty,0)$. Then
$$\lim_{n\rightarrow\infty}\frac{1}{n}\log\left[-\log \mathbb{P}\left(Z_n\left([\theta x^*n,+\infty)\right)<e^{an}\right)\right]=
\begin{cases}
\frac{(1-\theta)x^*}{L+x^*}\log b, & a\in[0,a^*);\cr
\bar{c}\log b, & a\in[a^*,\log m-I(\theta x^*)),
\end{cases}
$$
where $a^*:=\left[(\log m-I(x^*))\frac{L+\theta x^*}{L+x^*}+\frac{(1-\theta)x^*}{L+x^*}\log b\right]\wedge[\log m-I(\theta x^*)]$, and $\bar{c}\in(0,1)$ is the unique solution of
 $$\log m-I\left(\frac{\theta x^*+L\bar{c}}{1-\bar{c}}\right)-\frac{a-\bar{c}\log b}{1-\bar{c}}=0.$$
\end{theorem}
\begin{remark} If $I(x^*)=\log m$, then we have $a^*<\log m-I(\theta x^*)$. In fact, since $I(0)=0$ and $I(x)$ is convex, we have
$$I(\theta x^*)=I((1-\theta)0+\theta x^*)\leq (1-\theta)I(0)+\theta I(x^*)\leq \theta\log m.$$
Thus,
$$\frac{(1-\theta)x^*}{L+x^*}\log b+I(\theta x^*)<\frac{(1-\theta)x^*}{x^*}\log m+\theta\log m\leq \log m,$$
which implies $a^*<\log m-I(\theta x^*)$.
\end{remark}
\begin{remark}As we can see there is a phase transition at $a=a^*$. This is because $e^{a^*n}$ is the minimum population size that $b^{t_n}$ branching random walks located at $-Lt_n$ can contribute to $[\theta x^*n,\infty)$ in the remaining $n-t_n$ time. Thus when $a<a^*$, to achieve $\{Z_n\left([\theta x^*n,+\infty)\right)<e^{an}\}$ means there is no particles in $[\theta x^*n,+\infty)$ at time $n$. Hence, in this case, the lower deviation behaviour is exactly the same as the maximum of BRW (see \cite[Theorem 1.1]{Helower}).
\end{remark}
The following theorem considers the step size has Weibull tail. Recall that $\hat{c}$ is the unique solution of the equation
$\log m-I(\theta x^*+\hat{c})=a.$
\begin{theorem}\label{main3}(B\"{o}ttcher case, Weibull tail) Assume $p_1=0$, $\mathbb{P}(X\leq-x)=\Theta(1)e^{-\lambda x^{\alpha}}$ as $x\rightarrow+\infty$ for some $\alpha\in(0,+\infty)$ and $\lambda>0$. Then
$$\lim_{n\rightarrow\infty}\frac{1}{n^{\alpha}}\log\mathbb{P}\left(Z_n([\theta x^*n,+\infty))<e^{an}\right)=-\lambda C(x^*,\theta,b,\alpha,m,a),$$
where
$$
C(x^*,\theta,b,\alpha,m,a)
=\begin{cases}
b(x^{*}-\theta x^*)^{\alpha}, &{\alpha \in(0,1],~a\in[0,\log m-I(x^*)]}; \cr
b\hat{c}^{\alpha}, &{\alpha \in(0,1],~a\in(\log m-I(x^*),\log m-I(\theta x^*))};\cr
\left(b^{\frac{1}{\alpha-1}}-1\right)^{\alpha-1}(x^{*}-\theta x^*)^{\alpha},&{\alpha\in(1,+\infty),~a\in[0,\log m-I(x^*)]};\cr
\left(b^{\frac{1}{\alpha-1}}-1\right)^{\alpha-1}\hat{c}^{\alpha}, &{\alpha \in(1,+\infty),~a\in(\log m-I(x^*),\log m-I(\theta x^*))}.
\end{cases}
$$
\end{theorem}
\begin{remark}In fact, using the same idea of proving Theorem \ref{main3}, one could get the results below for Gumbel step size and Pareto step size.\\
If $p_1=0$, $\mathbb{P}(X\leq-x)=\Theta(1)\exp\{-e^{x^\alpha}\}$ as $x\rightarrow+\infty$ for some $\alpha>0$, then
$$
\lim_{n\rightarrow\infty}\frac{\log-\log\mathbb{P}\left(Z_n([\theta x^*n,+\infty))<e^{an}\right)}{n^{-\frac{\alpha}{1+\alpha}}}=
\begin{cases}
\left(\frac{\alpha}{1+\alpha}\log b\right)^{\frac{\alpha}{1+\alpha}}(x^*-\theta x^*)^{\frac{\alpha}{1+\alpha}}, &~a\in[0,\log m-I(x^*)];\cr
\left(\frac{\alpha}{1+\alpha}\log b\right)^{\frac{\alpha}{1+\alpha}}{\hat{c}}^{\frac{\alpha}{1+\alpha}}, &~a\in(\log m-I(x^*),\log m-I(\theta x^*)).
\end{cases}
$$
If $p_1=0$, $\mathbb{P}(X<-x)=\Theta(1)x^{-\alpha}$ as $x\rightarrow+\infty$ for some $\alpha>0$, then
$$\lim\limits_{n\rightarrow\infty}\frac{1}{\log n}\log\mathbb{P}\left(Z_n([\theta x^*n,\infty))<e^{an}\right)=-\alpha b.$$
\end{remark}
The rest of this paper is organised as follows. In Section \ref{sec1} and Section \ref{sec4}, we study the Schr\"oder case, where Theorem \ref{main1} and Theorem \ref{main4} are proved. Section \ref{sec1} considers the step size has a negative exponential moment. And Section \ref{sec4} study the case when the step size has heavy tails. Section \ref{sec2} and Section \ref{sec3} consider the B\"ottcher case, where Theorem \ref{main2} and Theorem \ref{main3} are proved. In Section \ref{sec2}, the step size is assumed to be bounded below. Section \ref{sec3} treat the Weibull step size, and this section is divided into two subsections: in the first section, we study the sub Weibull case; in the second section, we study the super Weibull case. As usual, $f(x)=\Theta(1)g(x)$ as $x\rightarrow+\infty$ means there exist constants $C\geq C'>0$ such that $C'\leq|f(x)/g(x)|\leq C$ for all $x>1$.

\section{Proof of Theorem \ref{main1}: Schr\"{o}der case, light tail}\label{sec1}
In this section, we are going to prove Theorem \ref{main1}. We first present several lemmas.
The following lemma can be found in \cite[Theorem 2]{Biggins1977}.
\begin{lemma} Assume $1<m<\infty$ and $\mathbb{E}[e^{\kappa X}]<\infty$ for some $\kappa>0$.\\
If $\theta\in[0,1)$, then
$$
\lim_{n\rightarrow\infty}\frac{1}{n}\log Z_n([\theta x^*n,\infty))=\log m-I(\theta x^*),~\mathbb{P}-a.s.~\text{on-non extinction},
$$
where $I(x):=\sup\limits_{t\in\mathbb{R}}\{tx-\log \mathbb{E}[e^{tX}]\}$, $x^*=\sup\{x\geq0:I(x)\leq\log m\}$.\\
If $\theta\in(1,+\infty)$, then
$$
\lim_{n\rightarrow\infty}Z_n([\theta x^*n,\infty))=0,~\mathbb{P}-a.s.
$$
\end{lemma}
The following remark is a direct consequence of above Lemma.
\begin{remark}\label{qswe} Let $\{\rho_n\}_{n\geq 1}$ be a sequence of numbers such that $\lim\limits_{n\rightarrow\infty}\rho_n=1$.\\
If $\theta\in[0,1)$, then
\begin{equation}\label{locala.s.}
 \lim_{n\rightarrow\infty}\frac{1}{n}\log Z_n([\rho_n\theta x^*n,\infty))=\log m-I(\theta x^*), ~\mathbb{P}-a.s.
 \end{equation}
If $\theta\in(1,+\infty)$, then
\begin{equation}\label{locala.s.0}
 \lim_{n\rightarrow\infty}Z_n([\rho_n\theta x^*n,\infty))=0,~\mathbb{P}-a.s.
\end{equation}
\end{remark}
\begin{remark} Biggins \cite{Biggins1979,Biggins1992} also studied the a.s. behaviour of $Z_n(\theta x^*n+B)$, where $B$ is a bounded measurable set.
\end{remark}
The next lemma is the well-known Cram\'er theorem; see \cite[Theorem 2.2.3]{Dembo}. Here and later, we define $S_n:=X_1+X_2+...+X_n$, where $X_i, i\geq 1$ are i.i.d copies of the step size $X$.
\begin{lemma} If $\mathbb{E}[e^{-\kappa X}]<\infty$ for some $\kappa>0$, then
$$\lim_{n\rightarrow\infty}\frac{1}{n}\log\mathbb{P}(S_n<-nx)=-I(-x).$$
\end{lemma}
The following lemma considers the asymptotic behavior of $I(x)$. Let $\lambda^*:=\sup\left\{\lambda\geq0:\mathbb{E}[e^{\lambda X}]<\infty\right\}$ and $\Lambda(\lambda):=\log \mathbb{E}\left[e^{\lambda X}\right]$.
\begin{lemma}\label{sratefunction} The rate function $I(x)$ can be classified into the following three cases.\\
(i) If $\Lambda'(\lambda)\uparrow+\infty$ as $\lambda\rightarrow \lambda^*$, then for any $x>0$, there exists some $\lambda\in(0,\lambda^*)$ such that
$$x=\Lambda'(\lambda),~I(x)=\lambda\Lambda'(\lambda)-\Lambda(\lambda).$$~~~~Furthermore, $\lim\limits_{x\rightarrow+\infty} \frac{I(x)}{x}=\lambda^*$.\\
(ii) If $\lambda^*=+\infty$, $\Lambda'(\lambda)$ converges to some finite limit as $\lambda\rightarrow\infty$, then $\Lambda'(\lambda)\uparrow \text{ess}\sup X$.  And

$$I(x)=
\begin{cases}
\text{positive finite}, &x\in(0,\text{ess}\sup X);\cr
-\log\mathbb{P}(X=\text{ess}\sup X), &x=\text{ess}\sup X;\cr
+\infty, &x\in (\text{ess}\sup X,\infty).
\end{cases}
$$
(iii) If $0<\lambda^*<+\infty$, $\Lambda'(\lambda)$ converges to some finite limit $T$ as $\lambda\rightarrow \lambda^*$, then $\mathbb{E}\left[e^{\lambda^*X}\right]<\infty$ and
$$I(x)=
\begin{cases}
\text{positive finite}, &x\in(0,T];\cr
\lambda^*x-\log \mathbb{E}\left[e^{\lambda^*X}\right], &x\in [T,\infty).
\end{cases}
$$
\end{lemma}
\begin{proof} In fact, all the statements above can be found in \cite[Sec 2.6]{durrett} except $\lim\limits_{x\rightarrow+\infty} \frac{I(x)}{x}=\lambda^*$. So, we just prove this. Note that, by the former part of (i) for any $x>0$, there exists $\lambda\in(0,\lambda^*)$ such that $x=\Lambda'(\lambda)$. It is easy to see that $x\rightarrow+\infty$ implies $\lambda\rightarrow \lambda^*$. Furthermore, $\Lambda(\cdot)$ is infinitely differentiable on $(0,\lambda^*)$ (see \cite[Exercise 2.2.24]{Dembo}). Thus, by L'Hospital's rule,
$$\lim\limits_{x\rightarrow+\infty} \frac{I(x)}{x}=\lim\limits_{\lambda\rightarrow \lambda^*} \frac{\lambda\Lambda'(\lambda)-\Lambda(\lambda)}{\Lambda'(\lambda)}=\lim\limits_{\lambda\rightarrow \lambda^*} \frac{\lambda\Lambda''(\lambda)+\Lambda'(\lambda)-\Lambda'(\lambda)}{\Lambda''(\lambda)}=\lambda^*.$$
\end{proof}

The next lemma considers the lower deviation probabilities of the branching process, which can be found in  \cite[Theorem 2.1]{GWlower}.
\begin{lemma}\label{branchldp}Let $\{a_n\}_{n\geq1}$ be a sequence of numbers such that $\lim\limits_{n\rightarrow\infty}a_n=+\infty$ and for any $t>0$, $\lim\limits_{n\rightarrow\infty} a_ne^{-tn}=0$.
$$\lim\limits_{n\rightarrow\infty} \frac{1}{n}\log \mathbb{P}(|Z_n|<a_n)=\log p_1.$$
\end{lemma}
Set
\begin{align}
f(\rho)&:=\log m-I\left(\frac{\theta x^*}{1-\rho}\right)-\frac{a}{1-\rho},~\bar{\rho}:=\sup\left\{\rho\in(0,1):f(\rho)\geq0\right\},\cr
g_{\rho}(h)&:=\log m-I\left(\frac{h+\theta x^*}{1-\rho}\right)-\frac{a}{1-\rho},~d:=d(\rho)=\sup\left\{h\in[0,+\infty):g_{\rho}(h)\geq0\right\}.\nonumber
\end{align}
\begin{lemma}\label{erdf} Assume $p_1>0$ and $\mathbb{E}\left[e^{\kappa X}\right]<+\infty$ for some $\kappa<0$. Then\\
(i) $\bar{\rho}\in(0,1)$, $d(\rho)\in(0,+\infty)$ for $\rho\in(0,\bar{\rho})$.\\
(ii) $g_{\rho}(d+\varepsilon)<0$ for $\varepsilon>0$, and $g_{\rho}(d-\varepsilon)>0$ for $\varepsilon\in(0,d).$\\
(iii) $d(0)>0$, $d(\bar{\rho})=0$ and $d(\rho)$ is decreasing on $(0,\bar{\rho}]$.\\
(iv) $\inf\limits_{\rho\in(0,\bar{\rho}]}\left\{\rho\log\frac{1}{p_1}+\rho I\left(\frac{-d}{\rho}\right)\right\}\in(0,+\infty)$.
\end{lemma}
\begin{proof}
Since $f(\rho)$ is strictly decreasing on $[0,1)$, and
$$~f(0)=\log m-I(\theta x^*)-a>0,~f(1-)=-\infty<0,$$
we have $\bar{\rho}\in(0,1)$. Fix $\rho\in(0,\bar{\rho})$.
 It is easy to see that $g_{\rho}(\cdot)$  is  decreasing on $[0,+\infty)$, and
 $$g_{\rho}(0)=f(\rho)>f(\bar{\rho})\geq0,~g_{\rho}(+\infty)=-\infty<0,$$
 where $f(\bar{\rho})\geq0$ follows by the fact that $I(\cdot)$ is lower semicontinuous. Therefore $d\in(0,+\infty)$ and $g_{\rho}(d)\geq 0$. Furthermore, $g_{\rho}(d+\varepsilon)<0$ for $\varepsilon>0$. In fact, if $g_{\rho}(d)=0$, since $g_{\rho}(\cdot)$ is strictly decreasing, we have
 $g_{\rho}(d+\varepsilon)< 0$. If $g_{\rho}(d)>0$, then $\frac{d+\theta x^*}{1-\rho}=\text{ess}\sup X<\infty$, which implies  $$g_{\rho}(d+\varepsilon)=-\infty< 0.$$ Thus (i) and (ii) hold. For (iii), $d(0)>0$ follows by the fact that $f(0)=\log m-I(\theta x^*)-a>0$ and $I(\cdot)$ is continuous at $\theta x^*$. We obviously have $d(\bar{\rho})=0$ if $f(\bar{\rho})=0$. If $f(\bar{\rho})>0$, then $\frac{\theta x^*}{1-\bar{\rho}}=\text{ess}\sup X<\infty$. Thus, $g_{\bar{\rho}}(\varepsilon)=-\infty$ for any $\varepsilon>0$, which implies $d(\bar{\rho})=0$. $d(\rho)$ is decreasing follows by the fact that
  $$g_{\rho'}(\cdot)>g_{\rho}(\cdot)~\text{for}~\rho'<\rho.$$
For (iv), since $\lim\limits_{\rho\rightarrow0}\frac{d(\rho)}{\rho}=+\infty$, by Lemma \ref{sratefunction}, we have
 $$\lim\limits_{\rho\rightarrow0}\rho I\left(\frac{-d}{\rho}\right)=\lim\limits_{\rho\rightarrow0}\frac{\rho}{d(\rho)}I\left(\frac{-d}{\rho}\right)d(\rho)>0.$$
 Thus, there exist constants $\delta,~\rho^*>0$ such that for every $0<\rho<\rho^*$,
  $$\rho\log\frac{1}{p_1}+\rho I\left(\frac{-d}{\rho}\right)>\delta.$$
 Furthermore, for $\bar{\rho}>\rho>\rho^*$,
 $$\rho\log\frac{1}{p_1}+\rho I\left(\frac{-d}{\rho}\right)>\rho^*\log\frac{1}{p_1}.$$
 Thus,
 \begin{equation}\label{uyi}
 \inf\limits_{\rho\in(0,\bar{\rho}]}\left\{\rho\log\frac{1}{p_1}+\rho I\left(\frac{-d}{\rho}\right)\right\}>\delta\wedge\left(\rho^*\log\frac{1}{p_1}\right)>0.
 \end{equation}
 On the other hand, since $\lim\limits_{\rho\rightarrow\bar{\rho}}\frac{d(\rho)}{\rho}=0$, we have
 \begin{equation}\label{iot}
 \lim\limits_{\rho\rightarrow\bar{\rho}}\rho\log\frac{1}{p_1}+\rho I\left(\frac{-d}{\rho}\right)=\bar{\rho}\log\frac{1}{p_1}.
 \end{equation}
 As a result, (\ref{uyi}), together with (\ref{iot}), implies (iii).
\end{proof}
Now we are ready to prove Theorem \ref{main1}: if $p_1>0$, $\mathbb{E}[e^{\kappa X}]<\infty$ for some $\kappa<0$. Then
$$\lim_{n\rightarrow\infty}\frac{1}{n}\log \mathbb{P}\left(Z_n([\theta x^*n,\infty))<e^{an}\right)
=-\inf\limits_{\rho\in(0,\bar{\rho}]}\left\{\rho\log\frac{1}{p_1}+\rho I\left(-\frac{d}{\rho}\right)\right\}.$$
For the lower bound, the strategy is to force single births up to time $\rho n$ (where $\rho\in(0,\bar{\rho}]$). Then, we move the single particle at time $\rho n$ to some position $-dn$ (where $d$ is chosen such that $Z_{n-\rho n}([dn+\theta x^*n,\infty))\approx e^{an}$). Optimizing for $\rho$ yields the desired lower bound. The proof of the upper bound goes by showing that the above strategy is optimal, thus we consider a series of intermediate times $\lfloor ni/k\rfloor$, $0<i/k<1$. To argue that sub-BRWs emanating from time $\lfloor ni/k\rfloor$ can easily produce $e^{an}$ descendants in $[\theta x^*n,\infty)$ at time $n$, we should insure that: there exist adequate particles at time $\lfloor ni/k\rfloor$; each individual at time $\lfloor ni/k\rfloor$ locates in a not very low position. This motivates the definitions of the following $\rho_n$ and $E$.
\begin{proof}
\textbf{Lower bound}.
 Fix $\varepsilon\in(0,1)$ and $\rho\in(0,\bar{\rho}]$. By Markov property, for $n$ large enough,
\begin{align}\label{uuuu}
\mathbb{P}\left(Z_n([\theta x^*n,\infty))<e^{an}\right)&\geq p^{\lfloor\rho n\rfloor}_1
\mathbb{P}\left(S_{\lfloor\rho n\rfloor}\leq -(d+\varepsilon)n\right)
\mathbb{P}\left(Z_{n-\lfloor\rho n\rfloor}\left([(d+\varepsilon+\theta x^*)n,\infty)\right)<e^{an}\right).
\end{align}
 Since $g_{\rho}(d+\varepsilon)<0$, by Remark \ref{qswe} and the dominated convergence theorem, for $n$ large enough,
\begin{align}
 &\mathbb{P}\left(Z_{n-\lfloor\rho n\rfloor}\left([(d+\varepsilon+\theta x^*)n,\infty)\right)<e^{an}\right)\cr
 &=\mathbb{P}\left(\frac{1}{n-\lfloor\rho n\rfloor}\log Z_{n-\lfloor\rho n\rfloor}\left(\left[\frac{(d+\varepsilon+\theta x^*)n}{n-\lfloor\rho n\rfloor}(n-\lfloor\rho n\rfloor),\infty\right)\right)-\frac{an}{n-\lfloor\rho n\rfloor}<0\right)\cr
 &\geq 0.9.\nonumber
 \end{align}
 Thus, plugging above into (\ref{uuuu}) yields that
 \begin{align}
 \mathbb{P}\left(Z_n([\theta x^*n,\infty))<e^{an}\right)&\geq 0.9p^{\lfloor\rho n\rfloor}_1
\mathbb{P}\left(\frac{S_{\lfloor\rho n\rfloor}}{\lfloor\rho n\rfloor}\leq -\frac{(d+\varepsilon)n}{\lfloor\rho n\rfloor}\right)\cr
&\geq0.9p^{\lfloor\rho n\rfloor}_1\exp\left\{-I\left(-\frac{d+2\varepsilon}{\rho}\right)\lfloor\rho n\rfloor\right\},\nonumber
 \end{align}
 where the last inequality comes from Cram\'er's theorem. As a consequence, for every $\rho\in(0,\bar{\rho}]$ and $\varepsilon \in(0,1)$,
$$\liminf\limits_{n\rightarrow\infty}\frac{1}{n}\log \mathbb{P}\left(Z_n([\theta x^*n,\infty))<e^{an}\right)\geq-\left\{\rho\log\frac{1}{p_1}+\rho I\left(-\frac{d+2\varepsilon}{\rho}\right)\right\}.$$
The desired lower bound follows by letting $\varepsilon\rightarrow0$ and optimizing $\rho$ on $(0,\bar{\rho}]$.

\bigskip
\textbf{Upper bound}. Fix $\delta\in(0,\bar{\rho})$, $k\geq \frac{1}{\bar{\rho}-\delta}$ (hence $\lfloor(\bar{\rho}-\delta)k\rfloor\geq1$). Since $d=d(\rho)$ is decreasing on $(0, \bar{\rho}]$ (see Lemma \ref{erdf}), there exists some constant $C_{\delta}>0$ such that for every $1\leq i\leq\lfloor(\bar{\rho}-\delta)k\rfloor$,
$$d\left(i/k\right)>C_{\delta}.$$
Fix $\varepsilon\in(0,C_{\delta})$. Set
$$\rho_n:=\sup\{\rho\in[0,1]:Z_{\lfloor\rho n\rfloor}\geq n^3\};~E:=\left\{\forall u\in Z_{\lfloor ni/k\rfloor},~S_u\in\left[-\left(d\left(i/k\right)-\varepsilon\right)n,+\infty\right) \right\}.$$
Let $Z^u_n$ be the $n$th generation of the sub-BRW emanating from particle $u$. By the branching property, we have
\begin{align}\label{erefdf}
&\mathbb{P}\left(Z_n([\theta x^*n,\infty))<e^{an}\right)\cr
&=\mathbb{P}\left(Z_n([\theta x^*n,\infty))<e^{an},\rho_n\geq\frac{\lfloor(\bar{\rho}-\delta)k\rfloor}{k}\right)
+\sum^{\lfloor(\bar{\rho}-\delta)k\rfloor}_{i=1}\mathbb{P}\left(Z_n([\theta x^*n,\infty))<e^{an},\frac{i-1}{k}\leq\rho_n<\frac{i}{k}\right)\cr
&\leq \sum^{\lfloor(\bar{\rho}-\delta)k\rfloor}_{i=1}\left[\mathbb{P}\left(\sum_{u\in Z_{\lfloor ni/k\rfloor}}Z^u_{n-{\lfloor ni/k\rfloor}}([\theta x^*n,\infty))<e^{an},~E,~\frac{i-1}{k}\leq\rho_n<\frac{i}{k}\right)
+\mathbb{P}\left(E^c,~\frac{i-1}{k}\leq\rho_n<\frac{i}{k}\right)\right]\cr
&~~~+\mathbb{P}\left(\rho_n\geq\frac{\lfloor(\bar{\rho}-\delta)k\rfloor}{k}\right).
\end{align}
 Let $Z^j_{n}$, $j\geq 1$ be i.i.d copies of $Z_{n}$. For the first term on the r.h.s of (\ref{erefdf}), by the branching property, there exists $N(k,\varepsilon)$ such that for all $1\leq i\leq\lfloor(\bar{\rho}-\delta)k\rfloor$ and $n\geq N(k,\varepsilon)$,
\begin{align}\label{su1}
&\mathbb{P}\left(\sum_{u\in Z_{\lfloor ni/k\rfloor}}Z^u_{n-{\lfloor ni/k\rfloor}}([\theta x^*n,\infty))<e^{an},~E,~\frac{i-1}{k}\leq\rho_n<\frac{i}{k}\right)\cr
&\leq \mathbb{P}\left(Z^j_{n-{\lfloor ni/k\rfloor}}\left(\left[\left(\theta x^*+d\left(i/k\right)-\varepsilon\right)n,\infty\right)\right)<e^{an}, 1\leq j\leq n^3\right)\cr
&\leq\mathbb{P}\left(\frac{1}{n-{\lfloor ni/k\rfloor}}\log \left(Z_{n-{\lfloor ni/k\rfloor}}\left(\left(\theta x^*+d\left(i/k\right)-\varepsilon\right)n+B\right)\right)<\frac{an}{n-{\lfloor ni/k\rfloor}}\right)^{n^3}\cr
&\leq e^{-n^3},
\end{align}
where the last inequality follows by (\ref{locala.s.}), the dominated convergence theorem and the fact that (see Lemma \ref{erdf} (ii))
$$
\log m-I\left(\frac{\theta x^*+d(i/k)-\varepsilon}{1-i/k}\right)-a>0.
$$
Let $h_{n,i}:={\lfloor ni/k\rfloor}-{\lfloor n(i-1)/k\rfloor}$. For the second term on the r.h.s of (\ref{erefdf}),  observe that
\begin{align}\label{t1}
&\mathbb{P}\left(\exists u\in Z_{\lfloor ni/k\rfloor},~S_u\leq-\left(d\left(i/k\right)-\varepsilon\right)n,~Z_{\lfloor n(i-1)/k\rfloor}\leq n^3\right)\cr
&\leq \mathbb{P}\left(\sum_{u\in Z_{\lfloor n(i-1)/k\rfloor}}Z^u_{h_{n,i}}\left(\left(-\infty,-\left(d\left(i/k\right)-\varepsilon\right)n\right]\right)>1;~Z_{\lfloor n(i-1)/k\rfloor}\leq n^3\right)\cr
&\leq \mathbb{P}\left(\sum^{n^3}_{j=1}Z^j_{h_{n,i}}\left(\left(-\infty,-\left(d\left(i/k\right)-\varepsilon\right)n-S^j_{\lfloor n(i-1)/k\rfloor}\right]\right)>1\right)\mathbb{P}\left(Z_{\lfloor n(i-1)/k\rfloor}\leq n^3\right)\cr
&\leq n^3 \mathbb{E}\left[Z_{h_{n,i}}\left(\left(-\infty,-\left(d\left(i/k\right)-\varepsilon\right)n-S_{\lfloor n(i-1)/k\rfloor}\right]\right)\right]\mathbb{P}\left(Z_{\lfloor n(i-1)/k\rfloor}\leq n^3\right),
\end{align}
where $S^j_n$, $j\geq 1$ are correlated random walks but have the same distribution as $S_n$, and the last inequality follows from Markov inequality. Let $\mathcal{F}:=\sigma\left(|Z_l|,1\leq l\leq h_{n,i}\right)$, since the branching and motion are independent, we have,
\begin{align}
&\mathbb{E}\left[Z_{h_{n,i}}\left(\left(-\infty,-\left(d\left(i/k\right)-\varepsilon\right)n-S_{\lfloor n(i-1)/k\rfloor}\right]\right)\right]\cr
&=\mathbb{E}\left[\sum_{u\in Z_{h_{n,i}}}\ind_{\left(-\infty,-\left(d\left(i/k\right)-\varepsilon\right)n\right]}\left(S_u+S_{\lfloor n(i-1)/k\rfloor}\right)\Bigg{|}\mathcal{F} \right]\cr
&=m^{{h_{n,i}}} \mathbb{P}\left(S_{\lfloor ni/k\rfloor}\leq -\left(d\left(i/k\right)-\varepsilon\right)n\right)\cr
&\leq m^{n/k+1}\exp\left\{-I\left(-\frac{d\left(i/k\right)-2\varepsilon}{i/k}\right)ni/k\right\},\nonumber
\end{align}
where the last inequality follows from the Cram\'er's theorem. Plugging above into (\ref{t1}) yields that
\begin{align}\label{su2}
\mathbb{P}\left(E^c,~\frac{i-1}{k}\leq\rho_n<\frac{i}{k}\right)&\leq n^3 m^{n/k+1} \exp\left\{-I\left(\frac{d\left(i/k\right)-2\varepsilon}{i/k}\right)ni/k\right\}\mathbb{P}\left(Z_{\lfloor n(i-1)/k\rfloor}\leq n^3\right)\cr
&\leq n^3m^{n/k+1}\exp\left\{(\log p_1+\varepsilon)ni/k-I\left(-\frac{d\left(i/k\right)-2\varepsilon}{i/k}\right)ni/k\right\},
\end{align}
where the last inequality follows from Lemma \ref{branchldp}. For the third term on the r.h.s of (\ref{erefdf}), again by Lemma \ref{branchldp}, there exists $N(\rho,\delta,\varepsilon)>0$ such that for $n\geq N(\rho,\delta,\varepsilon)$ and $k\geq 1$,
\begin{equation}\label{weerf}
\mathbb{P}\left(\rho_n\geq\frac{\lfloor(\bar{\rho}-\delta)k\rfloor}{k}\right)\leq \mathbb{P}\left(Z_{\lfloor{n\lfloor(\bar{\rho}-\delta)k\rfloor}/{k}\rfloor}\leq n^3\right)\leq
\mathbb{P}\left(Z_{\lfloor n(\bar{\rho}-\delta)\rfloor}\leq n^3\right)
\leq\exp\left\{(\log p_1+\varepsilon)(\bar{\rho}-\delta)n\right\}.
\end{equation}
 Plugging (\ref{su1}), (\ref{su1}) and (\ref{weerf}) into (\ref{erefdf}) yields that for fixed $\varepsilon,k$, $\rho$, $\delta$ there exists $N(\varepsilon,k,\rho,\delta)$ such that for $n>N(\varepsilon,k,\rho,\delta)$,
 \begin{align}
 &\mathbb{P}\left(Z_n([\theta x^*n,\infty))<e^{an}\right)\cr
 &\leq\sum^{\lfloor(\bar{\rho}-\delta)k\rfloor}_{i=1}2\exp\left\{(\log p_1+\varepsilon)ni/k-I\left(-\frac{d\left(i/k\right)-2\varepsilon}{i/k}\right)ni/k\right\}+\exp\left\{(\log p_1+\varepsilon)(\bar{\rho}-\delta)n\right\}.\nonumber
 \end{align}
 As a consequence,
 \begin{align}
 &\limsup_{n\rightarrow\infty}\frac{1}{n}\log\mathbb{P}\left(Z_n([\theta x^*n,\infty))<e^{an}\right)\cr
&\leq -\left[\min\limits_{1\leq i\leq \lfloor(\bar{\rho}-\delta)k\rfloor}\left\{(-\varepsilon-\log p_1)\frac{i}{k}-\frac{\log m}{k}+I\left(-\frac{d\left(i/k\right)-2\varepsilon}{i/k}\right)i/k\right\}\right]\wedge \left[\left(-\varepsilon-\log p_1\right)(\bar{\rho}-\delta)\right].\nonumber
\end{align}
Let $\varepsilon\rightarrow0$ and $\delta\rightarrow0$, above yields that
\begin{align}
&\limsup_{n\rightarrow\infty}\frac{1}{n}\log\mathbb{P}\left(Z_n([\theta x^*n,\infty))<e^{an}\right)\cr
&\leq-\left[\min\limits_{1\leq i\leq \lfloor\bar{\rho}k\rfloor}\left\{(-\log p_1)\frac{i}{k}-\frac{\log m}{k}+I\left(-\frac{d\left(i/k\right)}{i/k}\right)i/k\right\}\right]\wedge \left[\left(-\log p_1\right)\bar{\rho}\right]\cr
&= -\left[-\frac{\log m}{k}+\min\limits_{1\leq i\leq \lfloor\bar{\rho}k\rfloor}\left\{(-\log p_1)\frac{i}{k}+I\left(-\frac{d\left(i/k\right)}{i/k}\right)i/k\right\}\right]\wedge \left[\left(-\log p_1\right)\bar{\rho}\right]\cr
 &\leq-\left[-\frac{\log m}{k}+\inf\limits_{\rho\in[0,\bar{\rho})}\left\{(-\log p_1)\rho+I\left(-\frac{d}{\rho}\right)\rho\right\}\right]\wedge \left[\left(-\log p_1\right)\bar{\rho}\right].\nonumber
 \end{align}
 Then, let $k\rightarrow\infty$, we have
 \begin{align}
 &\limsup_{n\rightarrow\infty}\frac{1}{n}\log\mathbb{P}\left(Z_n([\theta x^*n,\infty))<e^{an}\right)\cr
 &\leq -\left[\inf\limits_{\rho\in[0,\bar{\rho})}\left\{(-\log p_1)\rho+I\left(-\frac{d}{\rho}\right)\rho\right\}\right]\wedge \left[\left(-\log p_1\right)\bar{\rho}\right]\cr
 &=-\inf\limits_{\rho\in(0,\bar{\rho}]}\left\{\rho\log\frac{1}{p_1}+\rho I\left(-\frac{d}{\rho}\right)\right\},\nonumber
 \end{align}
 where the last equality comes from the fact that $I\left(-\frac{d}{\bar{\rho}}\right)=I(0)=0$; see Lemma \ref{erdf} (iii).
\end{proof}
\section{Proof of Theorem \ref{main4}: Schr\"{o}der case, heavy tail}\label{sec4}
In this section, we are going to prove Theorem \ref{main4}. We first present several lemmas. The first lemma can be found in \cite[p40, Corollary 1]{Athreya}, which considers the asymptotic behaviour of branching process's generating function.
\begin{lemma}\label{genernatefun}If $0<s<1$, then
$$
\lim_{n\rightarrow\infty}\frac{\mathbb{E}[s^{|Z_n|}]}{p^n_1}=C_0,
$$
where $C_0$ is a positive constant depending only on $s$.
\end{lemma}
The next lemma concerns large deviation probabilities of sums of i.i.d Pareto tail random variables.
\begin{lemma}\label{pareto}Assume $\mathbb{P}(X<-x)=\Theta(1)x^{-\alpha}$ as $x\rightarrow+\infty$ for some $\alpha>0$. Then, there exists some constant $C_{\alpha}>0$ such that for all $x>0$, $n\geq 1$,
$$\mathbb{P}(S_n\leq -x)\leq C_{{\alpha}} n^2 x^{-\alpha}.$$
\end{lemma}
\begin{proof} According to \cite[Corollary 1.5]{Nagaev}: if $\mathbb{E}[|X|^t\ind_{\{X\leq 0\}}]<\infty$ for some $0<t\leq1$, then
\begin{equation}\label{wewe}
\mathbb{P}(S_n\leq -x)\leq n \mathbb{P}(X<-y)+\left(\frac{en\mathbb{E}[|X|^t\ind_{\{X\leq 0\}}]}{xy^{t-1}}\right)^{\frac{x}{y}}.
\end{equation}
Since $\mathbb{P}(X<x)=\Theta(1)x^{-\alpha}$, there exists some constant $C_1>0$ such that for all $x>0$,
$$  \mathbb{P}(X<-x)\leq C_1x^{-\alpha}.$$
Therefore, $\mathbb{E}[|X|^t\ind_{\{X\leq 0\}}]<\infty$ for $t=\frac{\alpha}{2}\wedge1$.
Let $y=\frac{t}{\alpha}x$, (\ref{wewe}) yields
\begin{align}
\mathbb{P}(S_n\leq x)&\leq n \mathbb{P}\left(X<-\frac{t}{\alpha}x\right)+\left(\frac{en\mathbb{E}[|X|^t\ind_{\{X\leq 0\}}]}{(\frac{t}{\alpha})^{t-1}x^t}\right)^{\frac{\alpha}{t}}\cr
&\leq C_1\left(\frac{t}{\alpha}\right)^{-\alpha}x^{-\alpha}n+\left(\frac{e\mathbb{E}[|X|^t\ind_{\{X\leq 0\}}]}{(\frac{t}{\alpha})^{t-1}}\right)^{\frac{\alpha}{t}}n^{\frac{\alpha}{t}}x^{-\alpha}\cr
&\leq C_{{\alpha}} n^2 x^{-\alpha}.\nonumber
\end{align}
\end{proof}

The following lemma gives an upper bound of large deviation probabilities of sums of i.i.d Weibull tail random variables, which is a direct consequence of \cite[Theorem 2.1]{FrankWebull}.
\begin{lemma}\label{webull} Assume $\mathbb{P}(X<-x)=\Theta(1)e^{-\lambda x^\alpha}$ as $x\rightarrow+\infty$ for some $\alpha\in(0,1)$. Let $\{t_n\}_{n\geq1}$ be a sequence of positive integer-valued numbers such than $t_n\rightarrow+\infty$ and $\lim\limits_{n\rightarrow\infty}\frac{t_n}{n}<\infty$. Then, for any $\varepsilon\in(0,1)$, $x>0$, for $n$ large enough,
$$\mathbb{P}(S_{t_n}\leq -xn)\leq e^{-(1-\varepsilon)\lambda x^{\alpha} n^{\alpha}}.$$
\end{lemma}
Let $G(c):=\log m-I(\theta x^*+c)-a$, $\hat{c}:=\inf\left\{c\geq 0:G(c)<0\right\}$.
\begin{lemma}\label{eqwen} \leavevmode\\
(i) If $a\in[0,\log m-I(x^*)]$, then $\hat{c}=(1-\theta)x^*$; If $a\in(\log m-I(x^*), \log m-I(\theta x^*)$, then $\hat{c}$ is the unique
~~~~~~~solution of $G(c)=0$ on $(0,(1-\theta)x^*)$.\\
(ii) If $\varepsilon>0$, then $G(\hat{c}+\varepsilon)<0$. If $\varepsilon\in(0,\hat{c})$, then $G(\hat{c}-\varepsilon)>0$.
\end{lemma}
\begin{proof}
If $a\in[0,\log m-I(x^*)]$, without loss of generality, we assume $\log m-I(x^*)>0$. By the definition of $x^*$, we have $x^*=\text{ess}\sup X\in(0,+\infty)$. Thus, by Lemma \ref{sratefunction} (ii), for $c>(1-\theta)x^*$,
 $G(c)=-\infty$. This, together with the fact that $G(\cdot)$ is decreasing on $[0,\infty)$, implies $\hat{c}=(1-\theta)x^*$.
If $a\in\left[\log m-I(x^*), \log m-I(\theta x^*)\right)$, then
  $$G(0)=\log m-I(\theta x^*)-a<0,~G((1-\theta)x^*)=\log m-I(x^*)-a>0.$$
Since $G(\cdot)$ is continuous and strictly decreasing on $(0,(1-\theta)x^*)$, thus $\hat{c}$ is the unique solution of $G(c)=0$. Thus (i) holds. By above arguments, we can obtain (ii) easily.
\end{proof}

\bigskip
Now, we are ready to prove Theorem \ref{main4}: assume $p_1=0$,\\
if $\mathbb{P}(X<-x)=\Theta(1)x^{-\alpha}$ as $x\rightarrow\infty$ for some $\alpha>0$, then
$$\lim\limits_{n\rightarrow\infty}\frac{1}{\log n}\log\mathbb{P}\left(Z_n([\theta x^*n,\infty))<e^{an}\right)=-\alpha;$$
if $\mathbb{P}(X\leq-x)=\Theta(1)e^{-\lambda x^{\alpha}}$ as $x\rightarrow\infty$ for some $\alpha\in(0,1)$ and $\lambda>0$, then
$$\lim\limits_{n\rightarrow\infty}\frac{1}{n^{\alpha}}\log\mathbb{P}\left(Z_n([\theta x^*n,\infty))<e^{an}\right)=
\begin{cases}
 -\lambda(1-\theta)x^*,  &a\in[0,\log m-I(x^*)];\\
-\lambda\hat{c}, &a\in(\log m-I(x^*),\log m-I(\theta x^*)).
\end{cases}
$$
For the lower bound, the strategy is to let the initial particle produces exactly one children and force its children to reach below some
$-\hat{c} n$ (where $\hat{c}$ is chosen such that $Z_{n-1}([\theta x^*n+\hat{c}n,\infty))\approx e^{an}$). For the upper bound, we first prove: under the event $\{Z_{t_n}([-\hat{c}n,\infty))/|Z_{t_n}|\geq 2/3\}$, it is very hard for $\{Z_n([\theta x^*n,\infty))<e^{an}\}$ to happen. Thus, the desired upper bound comes from the probability $$\mathbb{P}(Z_{t_n}((-\infty,-\hat{c}n))/|Z_{t_n}|> 1/3),$$ which is very easy to handle by a Markov inequality.
\begin{proof}
\textbf{Lower bound}. Since $\mathbb{P}(X<-x)=\Theta(1) x^{-\alpha}$ as $x\rightarrow\infty$, there exists a constant $C_2>0$ such that for $x>1$,
$$\mathbb{P}(X<-x)\geq C_2x^{-\alpha}.$$
Fix $\varepsilon>0$. By Markov property,
\begin{align}
\mathbb{P}\left(Z_n([\theta x^*n,\infty))<e^{an}\right)&\geq
\mathbb{P}\left(|Z_1|=1,S_{1}\leq -(\hat{c}+\varepsilon)n\right)\mathbb{P}\left(Z_{n-1}([(\theta x^*+\hat{c}+\varepsilon)n,\infty))<e^{an}\right)\cr
&\geq0.9p_1C_2(\hat{c}+\varepsilon)^{-\alpha}n^{-\alpha},
\end{align}
where the last inequality follows from the dominated convergence theorem and the fact that $G(\hat{c}+\varepsilon)<0$ (see Lemma \ref{eqwen}). As a result,
$$
\liminf\limits_{n\rightarrow\infty}\frac{1}{\log n}\log\mathbb{P}\left(Z_n([\theta x^*n,\infty))<e^{an}\right)\geq -\alpha.
$$

\bigskip
\textbf{Upper bound}. Let $t_n= \lfloor h\log n\rfloor$, where $h>\frac{\alpha}{-\log p_1}$. Recall that $Z^u_n$ is the $n$th generation of the sub-BRW emanating from particle $u$, and $Z^i_n$, $i\geq1$ are i.i.d copies of $Z_n$. Fix $\varepsilon\in(0, \hat{c})$. Observe that
\begin{align}\label{er3}
& \mathbb{P}\left(Z_n([\theta x^*n,\infty))<e^{an}\right)\cr
&\leq\mathbb{P}\left( Z_{t_n}([-(\hat{c}-\varepsilon)n,\infty))>\frac{2}{3}|Z_{t_n}| ,\sum_{u\in Z_{t_n}}Z^{u}_{n-t_n}([\theta x^*n,\infty))<e^{an}\right)+\mathbb{P}\left(\frac{1}{|Z_{t_n}|}Z_{t_n}([-(\hat{c}-\varepsilon)n,\infty))\leq\frac{2}{3}\right)\cr
&\leq\mathbb{P}\left(\forall 1\leq i< \frac{2}{3}|Z_{t_n}|,~Z^{i}_{n-t_n}([(\theta x^*+\hat{c}-\varepsilon)n,\infty))<e^{an}\right)
+\mathbb{P}\left(\frac{1}{|Z_{t_n}|}Z_{t_n}\left((-\infty,-(\hat{c}-\varepsilon)n)\right)>\frac{1}{3}\right).
\end{align}
Note that since $G(\hat{c}-\varepsilon)<0$ (see Lemma \ref{eqwen}), by (\ref{locala.s.}) and the dominated convergence theorem, for $n$ large enough,
$$
\mathbb{P}\left(Z_{n-t_n}([(\theta x^*+\hat{c}-\varepsilon)n,\infty))<e^{an}\right)<e^{-1}.
$$
Thus, for the first term on the r.h.s of (\ref{er3}), by Lemma \ref{genernatefun}, for $n$ large enough,
\begin{align}\label{er1}
&\mathbb{P}\left(\forall 1\leq i< \frac{2}{3}|Z_{t_n}|, Z^{i}_{n-t_n}([\theta x^*n,\infty))<e^{an}\right)\cr
&\leq\mathbb{E}\left[\mathbb{P}\left(Z_{n-t_n}([(\theta x^*+\hat{c}-\varepsilon)n,\infty))<e^{an}\right)^{\lfloor \frac{2}{3}|Z_{t_n}|\rfloor}\right]\cr
&\leq\mathbb{E}\left[e^{-\lfloor \frac{2}{3} |Z_{t_n}|\rfloor}\right]\leq 2C_0p^{h\log n}_1\leq 2C_0n^{-\alpha},
\end{align}
where the last inequality follows from the fact that $h>\frac{\alpha}{-\log p_1}$. Let $\mathcal{G}=\sigma(|Z_{t_n}|)$. For the second term on the r.h.s of (\ref{er3}), by Markov inequality,
\begin{align}\label{er2}
&\mathbb{P}\left(\frac{Z_{t_n}\left((-\infty,-(\hat{c}-\varepsilon)n)\right)}{Z_{t_n}}>\frac{1}{3}\right)\cr
&\leq 3\mathbb{E}\left[\frac{1}{Z_{t_n}}Z_{t_n}((-\infty,-(\hat{c}-\varepsilon)n))\right]\cr
&=3\mathbb{E}\left[\mathbb{E}\left[\frac{1}{Z_{t_n}}\sum\limits_{u\in Z_{t_n}}\ind_{(-\infty,-(\hat{c}-\varepsilon)n)}(S_u)\Bigg{|}\mathcal{G}\right]\right]\cr
&=3\mathbb{P}(S_{{\lfloor h\log n\rfloor}}<-(\hat{c}-\varepsilon)n)\leq3C_{\alpha} (h\log n)^2(\hat{c}-\varepsilon)^{-\alpha}n^{-\alpha},
\end{align}
where the last inequality comes from Lemma \ref{pareto}. Plugging (\ref{er1}) and (\ref{er2}) into (\ref{er3}) yields that
\begin{align}
\mathbb{P}\left(Z_n([\theta x^*n,\infty))<e^{an}\right)\leq 2C_0n^{-\alpha}+ 3C_{\alpha} (h\log n)^2(\hat{c}-\varepsilon)^{-\alpha}n^{-\alpha}.
\end{align}
As a result,
$$
\limsup\limits_{n\rightarrow\infty}\frac{1}{\log n}\log\mathbb{P}\left(Z_n([\theta x^*n,\infty))<e^{an}\right)\leq -\alpha.
$$
\par
For the Weibull case, the proof is similar, the mainly changes are to let $t_n=\lfloor h n^{\alpha}\rfloor$ in the upper bound (where $h>\frac{\lambda\hat{c}^{\alpha}}{-\log p_1}$) and use Lemma \ref{webull}. We feel free to omit its proof here.
\end{proof}
\section{Proof of Theorem \ref{main2}: B\"{o}ttcher case, bounded tail}\label{sec2}
In this section, we are going to prove Theorem \ref{main2}. We first present a lemma. Recall that
$$a^*=\left[(\log m-I(x^*))\frac{L+\theta x^*}{L+x^*}+\frac{(1-\theta)x^*}{L+x^*}\log b\right]\wedge[\log m-I(\theta x^*)].$$
 Let $$F_L(c):=\log m-I\left(\frac{\theta x^*+Lc}{1-c}\right)-\frac{a-c\log b}{1-c}=0.$$
\begin{lemma}\label{solution} Assume $a^*<\log m-I(\theta x^*)$. Then for any $a\in\left(a^*,\log m-I(\theta x^*)\right)$,\\
(i) the equation $F_L(c)=0$ has a unique solution on $\left(0,\frac{(1-\theta)x^*}{L+x^*}\right)$, denoted by $\bar{c}(L)$ (or $\bar{c}$);\\
(ii) $\bar{c}(L)$ is continuous w.r.t $L$, and
\begin{equation}\label{fdela}
F_L(\bar{c}(L)-\delta)>0~\text{for}~\delta\in(0,\bar{c}(L));~F_L(\bar{c}(L)+\delta)<0 ~\text{for}~\delta>0.
\end{equation}
\end{lemma}
\begin{proof} It is easy to see that $F_L(c)=\log m-\log b -I\left(\frac{\theta x^*+Lc}{1-c}\right)+\frac{\log b-a}{1-c}$, and $F_L(\cdot)$ is differentiable on $(0,\frac{(1-\theta)x^*}{L+x^*})$ (since $\frac{\theta x^*+Lc}{1-c}<x^*$). Thus,
$$F_L'(c)=-\frac{1}{(1-c)^2}\left[I'\left(\frac{\theta x^*+Lc}{1-c}\right)\left(L+\theta x^*\right)+a-\log b\right].$$
Since $I'(x)$ is increasing, $I'\left(\frac{\theta x^*+Lc}{1-c}\right)$ is increasing w.r.t $c$ on $(0,\frac{(1-\theta)x^*}{L+x^*})$. This implies the monotonicity of $F_L(c)$ has three cases: (1) increasing on $(0,\frac{(1-\theta)x^*}{L+x^*})$; (2) decreasing on $(0,\frac{(1-\theta)x^*}{L+x^*})$; (3) increasing on some $(0,h)$, then decreasing on $[h,\frac{(1-\theta)x^*}{L+x^*})$. Since $F(0)=\log m-I(\theta x^*)-a>0$, the monotonicity implies that $F_L(c)=0$ has at most one solution in all these three cases. On the other hand, since $a>a^*$, we have
$$ F_L\left(\frac{(1-\theta)x^*}{L+x^*}\right)=\log m-I(x^*)-\frac{a-\frac{(1-\theta)x^*}{L+x^*}\log b}{1-\frac{(1-\theta)x^*}{L+x^*}}<0.$$
Hence, $F_L(c)=0$ has a unique solution on $\left(0,\frac{(1-\theta)x^*}{L+x^*}\right)$. This concludes (i). For (ii), by implicit function theorem, it is easy to know that $\bar{c}(L)$ is continuous w.r.t $L$. (\ref{fdela}) follows be the definition of $\bar{c}(L)$.
\end{proof}

\bigskip
Now we are ready to prove  Theorem \ref{main2}: if $p_1=0$ and $\text{ess}\inf X=-L\in(-\infty,0)$, then
$$\lim_{n\rightarrow\infty}\frac{1}{n}\log\left[-\log \mathbb{P}\left(Z_n\left([\theta x^*n,+\infty)\right)<e^{an}\right)\right]=
\begin{cases}
\frac{(1-\theta)x^*}{L+x^*}\log b, & a\in[0,a^*);\cr
\bar{c}(L)\log b, & a\in[a^*,\log m-I(\theta x^*)).
\end{cases}
$$
For the lower bound, the strategy is to force every particle produces $b$ offsprings up to time $c^*(L)n$. Then let the displacement of every individual before time $c^*(L)n$ close enough to $-L$. Optimizing for $c^*(L)$ yields the desired lower bound. For the upper bound, our inspiration comes from the proof of Cram\'er's theorem. By the branching property, $Z_n([\theta x^*n,\infty))$ can be bounded below by $\sum^{b^{t_n}}_{i=1}Z^{(i)}_{n-t_n}([\theta x^*n+Lt_n,\infty))$ (sums of i.i.d copies of BRW). Thus, by Chebycheff's inequality, one can get an upper bound. To let this bound tends to zero, we use the dominated convergence theorem (which motivates us to define the following event $E_n$ and $K_n$ to satisfy the dominated convergence theorem's condition).
\begin{proof}
\textbf{Lower bound.} Let
$$
c^*(L)=
\begin{cases}
\bar{c}(L), & a^*<\log m-I(\theta x^*)~\text{and}~a\in(a^*,\log m-I(\theta x^*));\cr
\frac{(1-\theta)x^*}{L+x^*}, &\text{else}.
\end{cases}
$$
For every $L'\in(0,L)$ and $\delta\in(0,1-c^*(L'))$,
 let $t_n:=\lfloor (c^*(L')+\delta)n\rfloor$. Observe that
\begin{align}\label{bounde1}
\mathbb{P}\left(Z_n\left([\theta x^*n,\infty)\right)<e^{an}\right)&\geq \mathbb{P}\left(|Z_{t_n}|=b^{t_n}, \forall u\in Z_{t_n}, S_u\leq -L't_n, Z_n([\theta x^*n,\infty))<e^{an}\right)\cr
&\geq p_b^{\sum^{t_n-1}\limits_{k=0}b^k}\mathbb{P}(X\leq -L')^{\sum^{t_n-1}\limits_{k=0}b^k}\left[\mathbb{P}\left(Z_{n-t_n}([\theta x^*n+L't_n,\infty))<\frac{e^{an}}{b^{t_n}}\right)\right]^{b^{t_n}}.
\end{align}
We first consider the case that $a^*<\log m-I(\theta x^*)~\text{and}~a\in(a^*,\log m-I(\theta x^*))$. In this case, since $c^*(L)$ is continuous w.r.t $L$ and $c^*(L)<\frac{(1-\theta)x^*}{L+x^*}$ (see Lemma \ref{solution}), we have
$$\lim\limits_{L'\rightarrow L \atop \delta\rightarrow 0}\frac{\theta x^*+L'(c^*(L')+\delta)}{1-(c^*(L')+\delta)}=\frac{\theta x^*+Lc^*(L)}{1-c^*(L)}<x^*.$$
Hence, for $L'$ close enough to $L$ and $\delta>0$ small enough,
$$\frac{\theta x^*+L'(c^*(L')+\delta)}{1-(c^*(L')+\delta)}<x^*.$$
As a consequence, applying (\ref{locala.s.}), for $n$ large enough,
\begin{align}\label{bounde2}
&\mathbb{P}\left(Z_{n-t_n}([\theta x^*n+L't_n,\infty))<\frac{e^{an}}{b^{t_n}}\right)\cr
&\geq \mathbb{P}\left(\frac{1}{n-t_n}\log Z_{n-t_n}([\theta x^*n+L't_n,\infty))<\frac{an-t_n\log b}{n-t_n}\right)\geq0.9,
\end{align}
where the last inequality follows by the dominated
convergence theorem and the fact $F_{L'}(c^*(L')+\delta)<0$ (see Lemma \ref{solution}). Now we consider the case of
$a^*\geq\log m-I(\theta x^*)~\text{or}~a\in[0,a^*)$. In this case, since $c^*(L)=\frac{(1-\theta)x^*}{L+x^*}$, we have
$$\frac{(\theta x^*+L'(c^*(L')+\delta))}{1-(c^*(L')+\delta)}>x^*.$$
Thus, by (\ref{locala.s.0}), for $n$ large enough,
\begin{align}\label{bound3}
\mathbb{P}\left(Z_{n-t_n}([\theta x^*n+L't_n,\infty))<\frac{e^{an}}{b^{t_n}}\right)\geq 0.9.
\end{align}
Plugging (\ref{bounde2}) and (\ref{bound3}) into (\ref{bounde1}) yields that
$$\mathbb{P}\left(Z_n([\theta x^*n,\infty))<e^{an}\right)\geq p_b^{\sum^{t_n-1}\limits_{k=0}b^k}\mathbb{P}(X\leq -L')^{\sum^{t_n-1}\limits_{k=0}b^k}(0.9)^{b^{t_n}}\geq \left[0.9p_b\mathbb{P}(X\leq-L')\right]^{b^{t_n}},$$
where the last inequality follows by the fact that $\sum^{t_n-1}\limits_{k=0}b^k<b^{t_n}$. As a consequence, for $L'$ close enough to $L$ and $\delta$ small enough, we have
$$\limsup_{n\rightarrow\infty}\frac{1}{n}\log\left[-\log \mathbb{P}\left(Z_n([\theta x^*n,\infty))<e^{an}\right)\right]\leq (c^*(L')+\delta)\log b.$$
So, the desired lower bound follows by letting $L'\rightarrow L$ and $\delta\rightarrow 0$.

\bigskip
\textbf{Upper bound.} Let $t_n:=\lfloor (c^*(L)-\delta)n\rfloor$ for $\delta\in(0,c^*(L))$. Fix $\epsilon\in(0,1-\log2)$. By the branching property,
\begin{align}\label{sdsfs}
&\mathbb{P}\left(Z_n([\theta x^*n,\infty))<e^{an}\right)\cr
&=\mathbb{P}\left(Z_{t_n}([-Lt_n,+\infty))\geq b^{t_n},Z_n([\theta x^*n,\infty))<e^{an}\right)\cr
&\leq \mathbb{P}\left(\sum^{b^{t_n}}_{i=1}Z^{(i)}_{n-t_n}([\theta x^*n+Lt_n,\infty))<e^{an},|E_n|\geq \lfloor\epsilon b^{t_n}\rfloor \right)+\mathbb{P}\left(|E_n|< \lfloor\epsilon b^{t_n}\rfloor\right),
\end{align}
where $ Z^{(i)}_{n-t_n},~1\leq i\leq b^{t_n}$ are i.i.d copies of $Z_{n-t_n}$ and
$$E_n:=\left\{1\leq i\leq b^{t_n}:Z^{(i)}_{n-t_n}([\theta x^*n+Lt_n,\infty))\geq 1\right\}.$$
Set
$$K_n:=\left\{\forall 1\leq j\leq\lfloor\epsilon b^{t_n}\rfloor, Z^{(j)}_{n-t_n}([\theta x^*n+Lt_n,\infty))\geq 1 \right\}.$$
Observe that, by Stirling's formula $\lim\limits_{n\rightarrow\infty}\frac{n!}{\sqrt {2\pi n} (\frac{n}{e})^n}=1$, we have the following upper bound for combinatorial number: for $n$ large enough,
\begin{equation}\label{string}
\tbinom{b^{t_n}}{\lfloor\epsilon b^{t_n}\rfloor}\leq \exp\left\{-\left[(1-\epsilon)\log(1-\epsilon)+\epsilon\log\epsilon\right]b^{t_n}\right\}\leq e^{(\log2)b^{t_n}}.
\end{equation}
Note that $F_L((c^*(L)-\delta))>0$ (if $c^*(L)=\frac{(1-\theta)x^*}{L+x^*}$ this comes from the definition of $a^*$, if $c^*(L)=\bar{c}(L)$ this follows by Lemma \ref{solution} (ii)). For the first term on the r.h.s of (\ref{sdsfs}), for $n$ large enough,
\begin{align}\label{gfhy1}
&\mathbb{P}\left(\sum^{b^{t_n}}_{i=1}Z^{(i)}_{n-t_n}([\theta x^*n+Lt_n,\infty))<e^{an},|E_n|\geq \lfloor\epsilon b^{t_n}\rfloor \right)\cr
&\leq\tbinom{b^{t_n}}{\lfloor\epsilon b^{t_n}\rfloor}\mathbb{P}\left(\sum^{\lfloor\epsilon b^{t_n}\rfloor}_{j=1}\frac{1}{\lfloor\epsilon b^{t_n}\rfloor}Z^{(j)}_{n-t_n}([\theta x^*n+Lt_n,\infty))<\frac{e^{an}}{\lfloor\epsilon b^{t_n}\rfloor};K_n\right)\cr
&\leq \tbinom{b^{t_n}}{\lfloor\epsilon b^{t_n}\rfloor}\mathbb{P}\left(\sum^{\lfloor\epsilon b^{t_n}\rfloor}\limits_{j=1}\frac{1}{\lfloor\epsilon b^{t_n}\rfloor}\frac{1}{n-t_n}\log Z^{(j)}_{n-t_n}([\theta x^*n+Lt_n,\infty))<\frac{an-\log\lfloor\epsilon b^{t_n}\rfloor}{n-t_n};K_n\right)\cr
&\leq \tbinom{b^{t_n}}{\lfloor\epsilon b^{t_n}\rfloor}\mathbb{E}\left[\exp\left\{\frac{2}{\epsilon F_L((c^*(L)-\delta))}\left[\frac{an-\log\lfloor\epsilon b^{t_n}\rfloor}{n-t_n}\lfloor\epsilon b^{t_n}\rfloor-\sum^{\lfloor\epsilon b^{t_n}\rfloor}\limits_{j=1}\frac{1}{n-t_n}\log Z^{(j)}_{n-t_n}\left([\theta x^*n+Lt_n,\infty)\right)\right]\right\};K_n\right]\cr
&=\tbinom{b^{t_n}}{\lfloor\epsilon b^{t_n}\rfloor}\mathbb{E}^{\lfloor\epsilon b^{t_n}\rfloor}\Bigg[\exp\left\{\frac{2}{\epsilon F_L((c^*(L)-\delta))}\left[\frac{an-\log \lfloor\epsilon b^{t_n}\rfloor}{n-t_n}-\frac{1}{n-t_n}\log Z_{n-t_n}\left([\theta x^*n+Lt_n,\infty)\right)\right]\right\}\cr
&~~~~~~~~~~~~~~~~~~~~~~~~~;Z_{n-t_n}\left([\theta x^*n+Lt_n,\infty)\right)\geq 1\Bigg],
\end{align}
where the second inequality follows by the concavity of $g(x)=\log x$, the third inequality comes from Chebycheff's inequality. Note that
 $$\frac{\theta x^*+L(c^*(L)-\delta)}{1-(c^*(L)-\delta)}<x^*.$$
Thus, by (\ref{locala.s.}),
$$\lim\limits_{n\rightarrow\infty}\frac{2}{\epsilon F_L((c^*(L)-\delta))}\left[\frac{an-\log \lfloor\epsilon b^{t_n}\rfloor}{n-t_n}-\frac{1}{n-t_n}\log Z_{n-t_n}\left([\theta x^*n+Lt_n,\infty)\right)\right]=\frac{2}{\epsilon},~\mathbb{P}-a.s.$$
Therefore, by the dominated convergence theorem and (\ref{string}), (\ref{gfhy1}) yields
\begin{align}\label{tbb1}
\mathbb{P}\left(\sum^{b^{t_n}}_{i=1}Z^{(i)}_{n-t_n}([\theta x^*n+Lt_n,\infty))<e^{an},|E_n|\geq \lfloor\epsilon b^{t_n}\rfloor \right)\leq e^{-(1-\log 2)b^{t_n}}.
\end{align}
For the second term on the r.h.s of (\ref{sdsfs}), for $n$ large enough,
\begin{align}\label{gfhy2}
\mathbb{P}\left(|E_n|<\lfloor\epsilon b^{t_n}\rfloor\right)&\leq \tbinom{b^{t_n}}{\lfloor(1-\epsilon)b^{t_n}\rfloor}\mathbb{P}^{\lfloor(1-\epsilon)b^{t_n}\rfloor}\left(Z_{n-t_n}\left([\theta x^*n+Lt_n,\infty)\right)=0\right)\cr
&\leq e^{(\log 2)b^{t_n}}\mathbb{P}^{\lfloor(1-\epsilon)b^{t_n}\rfloor}\left(Z_{n-t_n}\left([\theta x^*n+Lt_n,\infty)\right)=0\right)\cr
&\leq e^{-\left[(1-\epsilon)-\log2\right]b^{t_n}},
\end{align}
where the second inequality follows from (\ref{string}), and the third inequality follows from (\ref{locala.s.}). Plugging (\ref{tbb1}) and (\ref{gfhy2}) into (\ref{sdsfs}) yields that for $\epsilon\in(0,1-\log2)$ and $n$ large enough,
$$\mathbb{P}\left(Z_n([\theta x^*n,\infty))<e^{an}\right)\leq e^{-(1-\log 2)b^{t_n}}+e^{-\left[(1-\epsilon)-\log2\right]b^{t_n}}\leq2e^{-\left[(1-\epsilon)-\log2\right]b^{t_n}}.$$
Therefore,
$$\liminf_{n\rightarrow\infty}\frac{1}{n}\log\left[-\log \mathbb{P}\left(Z_n([\theta x^*n,\infty))<e^{an}\right)\right]\geq (c^*(L)-\delta)\log b.$$
So, the desired upper bound follows by letting $\delta\rightarrow 0$.
\end{proof}
\section{Proof of Theorem \ref{main3}: B\"{o}ttcher case, Weibull tail}\label{sec3}
In this section, we assume $p_1=0$ and $\mathbb{P}(X\leq-z)=\Theta(1)e^{-\lambda z^{\alpha}}$ as $z\rightarrow\infty$ for some $\alpha\in(0,+\infty)$ and $\lambda>0$. This section will be divided into two sub sections. In the first subsection, we consider the sub Weibull case (i.e. $\alpha\in(0,1]$). In the second subsection, we consider the super Weibull case (i.e. $\alpha>1$).\par
\subsection{Sub Weibull tail: $\alpha\in(0,1]$}
In this subsection, we consider the sub Weibull tail case, i.e., $\mathbb{P}(X\leq-z)=\Theta(1)e^{-\lambda z^{\alpha}}$ for some $\alpha\in(0,1]$. We are going to prove the following:
$$\lim_{n\rightarrow\infty}\frac{1}{n^{\alpha}}\log\mathbb{P}\left(Z_n([\theta x^*n,+\infty))<e^{an}\right)=
\begin{cases}
-\lambda b(x^{*}-\theta x^*)^{\alpha}, &{\alpha \in(0,1],~a\in[0,\log m-I(x^*)]}; \cr
-\lambda b\hat{c}^{\alpha}, &{\alpha \in(0,1],~a\in(\log m-I(x^*),\log m-I(\theta x^*))},
\end{cases}
$$
where $\hat{c}$ is the unique solution of $\log m-I(\theta x^*+\hat{c})=a.$\par
For the lower bound, the strategy is to let the initial particle produces exactly $b$ children and force these children to reach below
$-\hat{c} n$ (where $\hat{c}$ is chosen such that $Z_{n-1}([\theta x^*n+\hat{c}n,\infty))\approx e^{an}$). For the upper bound, we first prove: under the event $\{Z_{t_n}([-\hat{c}n,\infty))\geq b^{t_n}\}$, it is very hard for $\{Z_n([\theta x^*n,\infty))<e^{an}\}$ to happen. Thus, the leading order of $\mathbb{P}\left(Z_n([\theta x^*n,+\infty))<e^{an}\right)$ comes from the probability $\mathbb{P}(Z_{t_n}([-\hat{c}n,\infty))<b^{t_n})$, which has been already considered in \cite{Helower}.
\begin{proof}
\textbf{Lower bound.} Recall that $G(c)=\log m-I(\theta x^*+c)-a$, $\hat{c}=\inf\left\{c\geq 0:G(c)<0\right\}$. We first consider the case
 $a>\log m-I(x^*)$. In this case, we have shown (see Lemma \ref{eqwen}):
 $$\hat{c}~\text{is the unique solution of}~G(c)=0~\text{on}~(0,(1-\theta)x^*);~G(\hat{c}+\delta)<0~\text{for}~ \delta>0.$$
Since $\mathbb{P}(X\leq-z)=\Theta(1)e^{-\lambda z^{\alpha}}$, there exists some constant $C_4>0$ such that for all $z>0$,
  \begin{equation}\label{c4}
  \mathbb{P}(X\leq-z)\geq C_4 e^{-\lambda z^{\alpha}}.
  \end{equation}
  The branching property  gives, for any $\delta>0$ and $n$ large enough,
\begin{align}
\mathbb{P}\left(Z_n([\theta x^*n,\infty))<e^{an}\right)&\geq \mathbb{P}\left(Z_1=b,\forall u\in Z_1,S_u\leq-(\hat{c}+\delta)n,\sum^b_{i=1}Z^{(i)}_{n-1}([(\theta x^*+\hat{c}+\delta)n,\infty))<e^{an}\right)\cr
&\geq \mathbb{P}\left(Z_1=b,\forall u\in Z_1,S_u\leq-(\hat{c}+\delta)n\right)\mathbb{P}\left(Z^{(i)}_{n-1}([(\theta x^*+\hat{c}+\delta)n,\infty))<e^{an}/b,~\forall 1\leq i\leq b\right)\cr
&\geq p_b C^b_4e^{-\lambda (\hat{c}+\delta)^{\alpha}n^{\alpha}b} \mathbb{P}\left(Z_{n-1}([(\theta x^*+\hat{c}+\delta)n,\infty))<e^{an}/b\right)^b\cr
&\geq p_bC^b_4e^{-\lambda (\hat{c}+\delta)^{\alpha}n^{\alpha}b}\mathbb{P}\left(\frac{\log Z_{n-1}([(\theta x^*+\hat{c}+\delta)n,\infty))}{n-1}<\frac{an-\log b}{n-1}\right)^b\cr
&\geq 0.9p_bC^b_4e^{-\lambda (\hat{c}+\delta)^{\alpha}n^{\alpha}b},
\end{align}
where the last inequality follows by (\ref{locala.s.}) and the fact $G(\hat{c}+\delta)<0$. Thus, for any $\delta>0$,
$$\liminf_{n\rightarrow\infty}\frac{1}{n^{\alpha}}\log\mathbb{P}\left(Z_n([\theta x^*n,+\infty))<e^{an}\right)\geq-\lambda (\hat{c}+\delta)^{\alpha}b,$$
So the desired lower bound follows by letting $\delta\rightarrow0$.

\bigskip
\textbf{Upper bound.} Let $t_n:=\lfloor \frac{2\alpha}{\log b}\log n\rfloor$. Fix $\delta\in(0,\hat{c})$. Note that
\begin{align}
&\mathbb{P}\left(Z_n([\theta x^*n,\infty))<e^{an}\right)\cr
&\leq \mathbb{P}\left(Z_{t_n}([-(\hat{c}-\delta)n,\infty))\geq b^{t_n-1},~Z_n([\theta x^*n,\infty))<e^{an}\right)+\mathbb{P}\left(Z_{t_n}([-(\hat{c}-\delta)n,\infty))< b^{t_n-1}\right).\nonumber
\end{align}
For the first term on the right hand side above, note that by the branching property, for $n$ large enough,
\begin{align}
\mathbb{P}\left(Z_{t_n}([-(\hat{c}-\delta)n,+\infty))\geq b^{t_n-1},~Z_n([\theta x^*n,\infty))<e^{an}\right)&\leq
\mathbb{P}\left(\sum^{b^{t_n-1}}_{k=1}Z^{(i)}_{n-t_n}([(\theta x^*+\hat{c}-\delta)n,\infty))<e^{an}\right)\cr
&\leq\mathbb{P}\left(\frac{\log Z_{n-t_n}([(\theta x^*+\hat{c}-\delta)n,\infty))}{n-t_n}<\frac{an}{{n-t_n}}\right)^{b^{t_n-1}}\cr
&\leq(0.9)^{b^{t_n-1}},\nonumber
\end{align}
where the last inequality follows by (\ref{locala.s.}) and the fact that $G(\hat{c}-\delta)>0$ (see Lemma \ref{eqwen}). For the second term, since $|Z_1|\geq b$, for any $\epsilon\in(0,1)$ and $n$ large enough, we have
\begin{align}
\mathbb{P}\left(Z_{t_n}([-(\hat{c}-\delta)n,\infty))< b^{t_n-1}\right)&\leq\mathbb{P}\left(\forall u\in Z_1, Z^u_{t_n-1}\left(\left(-\infty,-(\hat{c}-\delta)n\right)\right)\geq1\right)\cr
&\leq \mathbb{P}\left(Z_{t_n-1}\left(\left(-\infty,-(\hat{c}-\delta)n-X'\right)\right)\geq1\right)^b\cr
&\leq \mathbb{E}\left[Z_{t_n-1}\left(\left(-\infty,-(\hat{c}-\delta)n-X'\right)\right)\right]^b\cr
&=\left[m^{(t_n-1)}\mathbb{P}\left(S_{t_n}\leq-(\hat{c}-\delta)n\right)\right]^b\cr
&\leq m^{b(t_n-1)}e^{-\lambda(1-\epsilon) (\hat{c}-\delta)^{\alpha}n^{\alpha}b},
\end{align}
where $X'$ is a copy of $X$ and independent of $Z_{t_n-1}$ and the last inequality follows from Lemma \ref{webull}. Hence, for $\delta\in(0,\hat{c})$ and $\epsilon\in(0,1)$,
$$
\mathbb{P}\left(Z_n([\theta x^*n,\infty))<e^{an}\right)\leq m^{b(t_n-1)}e^{-\lambda(1-\epsilon) (\hat{c}-\delta)^{\alpha}n^{\alpha}b}+(0.9)^{b^{t_n}}\leq 2m^{b(t_n-1)}e^{-\lambda(1-\epsilon) (\hat{c}-\delta)^{\alpha}n^{\alpha}b},
$$
which implies the desired upper bound.\par
 Now, it remains to deal with the case that $\log m-I(x^*)>0$ and $a\in[0,\log m-I(x^*)]$. In this case, $\hat{c}=(1-\theta)x^*$. The proof is basically the same as above. The only difference is to replace (\ref{locala.s.}) with (\ref{locala.s.0}).
\end{proof}
\subsection{Super Weibull tail: $\alpha>1$}
In this subsection, we consider the super Weibull tail case, i.e., $\mathbb{P}(X\leq-z)=\Theta(1)e^{-\lambda z^{\alpha}}$ for some $\alpha\in(1,+\infty)$. We are going to prove the following:\\
if $a\in[0,\log m-I(x^*)]$, then

$$\lim_{n\rightarrow\infty}\frac{1}{n^{\alpha}}\log\mathbb{P}\left(Z_n([\theta x^*n,+\infty))<e^{an}\right)=\left(b^{\frac{1}{\alpha-1}}-1\right)^{\alpha-1}(x^{*}-\theta x^*)^{\alpha};$$

\noindent if $a\in(\log m-I(x^*),\log m-I(\theta x^*))$, then
$$\lim_{n\rightarrow\infty}\frac{1}{n^{\alpha}}\log\mathbb{P}\left(Z_n([\theta x^*n,+\infty))<e^{an}\right)=\left(b^{\frac{1}{\alpha-1}}-1\right)^{\alpha-1}\hat{c}^{\alpha}.$$
 The proof of the lower bound that is given below uses a method similar to that of \cite[Theorem 1.3]{Helower}. We force the branching random walk to behave a $b$-regular tree up to some time $t_n=\Theta(1)\log n$. Then, we let each individual on this regular tree to make displacement smaller than some $-a_k$ (where $k$ is the generation of the individual). Obviously, to make the lower bound as large as possible, $\{a_k\}_{k\geq 1}$ should be the solution of $$\inf\limits_{\sum^{t_n}\limits_{i=1}a_k=\hat{c}n\atop a_k\geq 0, k\geq 1}\sum^{t_n}_{k=1}a^{\alpha}_kb^k.$$
 This help us to find the right $\{a_k\}_{k\geq 1}$. For the upper bound, we use the same method as sub Weibull case.

\begin{proof}
\textbf{Lower bound.} Recall that $G(c)=\log m-I(\theta x^*+c)-a$, $\hat{c}=\inf\left\{c\geq 0:G(c)<0\right\}$. Again, we first consider the case
 $a>\log m-I(x^*)$. Set $t_n:=\lfloor \frac{\alpha}{2\log b}\log n\rfloor$, $b_{\alpha}:=b^{\frac{1}{\alpha-1}}$ and $a_k:=\frac{b_{\alpha}-1}{b^k_{\alpha}}(\hat{c}+\delta)n$, where $\delta$ is a positive constant. By simple calculation, for $n$ large enough,
\begin{equation}\label{akgj}
(\hat{c}+\delta/2)n\leq\sum^{t_n}\limits_{k=1}a_{k}\leq(\hat{c}+\delta)n,~\sum^{t_n}\limits_{k=1}a^{\alpha}_{k}b^k\leq \left(b^{\frac{1}{\alpha-1}}-1\right)^{\alpha-1}(\hat{c}+\delta)^{\alpha}n^{\alpha}.
\end{equation}
Denote by $X_u$ the displacement of particle $u$. By the branching property, for $n$ large enough,
\begin{align}\label{sadasf}
&\mathbb{P}\left(Z_n([\theta x^*n,\infty))<e^{an}\right)\cr
&\geq \mathbb{P}\left(|Z_{t_n}|=b^{t_n};\forall u\in Z_{i}, 1\leq i\leq t_n, X_u\leq -a_{|u|}; \sum^{b^{t_n}}\limits_{i=1}Z^{(i)}_{n-t_n}\left(\left[\theta x^*n+\sum^{t_n}\limits_{k=1}a_{k},\infty\right)\right)<e^{an}\right)\cr
&\geq p_b^{\sum^{t_n-1}\limits_{k=0}b^k}\prod^{t_n}_{k=1}\mathbb{P}(X_k\leq -a_k)^{b^k}\left[\mathbb{P}\left(Z_{n-t_n}\left(\left[\theta x^*n+\sum^{t_n}\limits_{k=1}a_{k},\infty\right)\right)<\frac{e^{an}}{b^{t_n}}\right)\right]^{b^{t_n}}\cr
&\geq p_b^{\sum^{t_n-1}\limits_{k=0}b^k}C^{\sum^{t_n}\limits_{k=1}b^k}_4
\exp\left\{-\lambda\sum^{t_n}_{k=1}a^{\alpha}_kb^k\right\}\left[\mathbb{P}\left(\frac{1}{n-t_n}\log Z_{n-t_n}\left(\left[\theta x^*n+(\hat{c}+\delta/2)n,\infty\right)\right)<\frac{an-t_n\log b}{n-t_n}\right)\right]^{b^{t_n}},
\end{align}
where $C_4$ is defined in (\ref{c4}). Since (\ref{locala.s.}) and $G(\hat{c}+\delta/2)<0$ (see Lemma \ref{eqwen}), for $n$ large enough, we have
\begin{equation}\label{dsfasd}
\mathbb{P}\left(\frac{\log Z_{n-t_n}\big(\left[(\theta x^*+\hat{c}+\delta/2)n,\infty\right)\big)}{n-t_n}<\frac{an-t_n\log b}{n-t_n}\right)\geq 0.9.
\end{equation}
So, plugging (\ref{akgj}) and (\ref{dsfasd}) into (\ref{sadasf}) yields that
\begin{align}
\mathbb{P}\left(Z_n(\theta x^*n+B)<e^{an}\right)&\geq p_b^{\sum^{t_n-1}\limits_{k=0}b^k}C^{\sum^{t_n}\limits_{k=1}b^k}_4(0.9)^{b^{t_n}}
\exp\left\{-\lambda\sum^{t_n}_{k=1}a^{\alpha}_kb^k\right\}^{b^{t_n}}\cr
&\geq (0.9C_4p_b)^{b^{t_n+1}} \exp\left\{-\lambda \left(b^{\frac{1}{\alpha-1}}-1\right)^{\alpha-1}(\hat{c}+\delta)^{\alpha}n^{\alpha}\right\}\cr
&\geq (0.9C_4p_b)^{bn^{\alpha/2}}\exp\left\{-\lambda \left(b^{\frac{1}{\alpha-1}}-1\right)^{\alpha-1}(\hat{c}+\delta)^{\alpha}n^{\alpha}\right\},
\end{align}
where the second inequality use the fact $\sum^{t_n}\limits_{k=0}b^k<b^{t_n+1}$. Thus, for every small $\delta>0$, we have
$$\liminf_{n\rightarrow\infty}\frac{1}{n^{\alpha}}\log\mathbb{P}\left(Z_n([\theta x^*n,+\infty)<e^{an}\right)\geq-\lambda\left(b^{\frac{1}{\alpha-1}}-1\right)^{\alpha-1}(\hat{c}+\delta)^{\alpha}.$$
The desired lower bound finally follows by letting $\delta\rightarrow 0$.

\bigskip
\textbf{Upper bound.} The upper bound is similar to the sub Weibull tail case. Set $t_n:=\lfloor \frac{2\alpha}{\log b}\log n\rfloor$, $\delta_n:=\lfloor \frac{\alpha}{2\log b}\log n\rfloor$. For $\delta\in(0,\hat{c})$, we have
\begin{align}
&\mathbb{P}\left(Z_n([\theta x^*n,\infty))<e^{an}\right)\cr
&\leq \mathbb{P}\left(Z_{t_n}([-(\hat{c}-\delta)n,\infty))\geq b^{t_n-\delta_n},~Z_n([\theta x^*n,\infty))<e^{an}\right)+\mathbb{P}\left(Z_{t_n}([-(\hat{c}-\delta)n,\infty))< b^{t_n-\delta_n}\right).\nonumber
\end{align}
For the first term on the r.h.s above, note that by the branching property and (\ref{locala.s.}), we have
$$
\mathbb{P}\left(Z_{t_n}([-(\hat{c}-\delta)n,\infty))\geq {b^{t_n-\delta_n}},~Z_n([\theta x^*n,\infty))<e^{an}\right)\leq(0.9)^{b^{t_n-\delta_n}}.
$$
For the second term, \cite[(3.36)]{Helower} gives that for any $\epsilon>0$ and $n$ large enough,
$$
\mathbb{P}\left(Z_{t_n}([-(\hat{c}-\delta)n,+\infty))< b^{t_n-\delta_n}\right)\leq \exp\left\{-\lambda(1-\epsilon)\left(b^{\frac{1}{\alpha-1}}-1\right)^{\alpha-1}(1-b^{-\delta_n})^{1+\alpha}(\hat{c}-\delta)^{\alpha}n^{\alpha}\right\}.
$$
Thus, for any small $\epsilon>$ and $\delta>0$, we have
$$\limsup_{n\rightarrow\infty}\frac{1}{n^{\alpha}}\log\mathbb{P}\left(Z_n([\theta x^*n,+\infty))<e^{an}\right)\geq-\lambda(1-\epsilon)\left(b^{\frac{1}{\alpha-1}}-1\right)^{\alpha-1}(\hat{c}-\delta)^{\alpha},$$
which implies the result.\par
It remains to treat the case of $\log m-I(x^*)>0$ and $a\in[0,\log m-I(x^*)]$. In this case, one can show that $\hat{c}=(1-\theta)x^*$. The proof is basically the same as above. The only difference is to replace (\ref{locala.s.}) with (\ref{locala.s.0}).
\end{proof}
\textbf{Acknowledgements}
I am grateful to my supervisor Hui He for his constant help while working
on this subject, all the useful discussions and advices. I also would like to thank Lianghui Luo for interesting
discussions.

\vspace{0.2cm}
Shuxiong Zhang\\
School of Mathematical Sciences, Beijing Normal University, Beijing 100875, People's Republic of China.\\
E-mail: shuxiong.zhang@qq.com.
\end{document}